\RequirePackage{plautopatch}
\documentclass[11pt]{amsart}

\usepackage[a4paper,margin=28mm]{geometry}
\usepackage{amsmath,amssymb,amsthm,mathtools}
\usepackage{microtype,graphicx,scalerel,xcolor}
\usepackage[pdfusetitle]{hyperref}

\hypersetup{
    colorlinks=true,
}

\usepackage{tikz}
\usetikzlibrary{positioning}
\usetikzlibrary{svg.path}
\definecolor{orcid_color}{HTML}{A6CE39}

\hypersetup{pdfborder={0 0 0}} 
\DeclareRobustCommand{\orcidicon}{%
	\raisebox{.2mm}{\scalerel*{%
	\begin{tikzpicture}[xscale=1,yscale=-1,transform shape]
	\filldraw[color=orcid_color] svg {M256,128c0,70.7-57.3,128-128,128C57.3,256,0,198.7,0,128C0,57.3,57.3,0,128,0C198.7,0,256,57.3,256,128z};
	\filldraw[color=white] svg {M86.3,186.2H70.9V79.1h15.4v48.4V186.2z} svg {M108.9,79.1h41.6c39.6,0,57,28.3,57,53.6c0,27.5-21.5,53.6-56.8,53.6h-41.8V79.1z M124.3,172.4h24.5
		c34.9,0,42.9-26.5,42.9-39.7c0-21.5-13.7-39.7-43.7-39.7h-23.7V172.4z} svg {M88.7,56.8c0,5.5-4.5,10.1-10.1,10.1c-5.6,0-10.1-4.6-10.1-10.1c0-5.6,4.5-10.1,10.1-10.1
		C84.2,46.7,88.7,51.3,88.7,56.8z};
	\end{tikzpicture}}{|}}%
}
\newcommand{\orcid}[1]{\href{https://orcid.org/#1}{\orcidicon}}

\def\E{\mathbb E}
\def\P{\mathbb P}
\def\AA{\mathcal A}
\def\FF{\mathcal F}
\def\GG{\mathcal G}

\newtheorem{theorem}{Theorem}
\newtheorem{lemma}{Lemma}
\newtheorem{proposition}[lemma]{Proposition}

\title{A product inequality and its application to cross-intersecting families} 
\author[Toshihiro Shimizu]{Toshihiro Shimizu\,\orcid{0009-0009-1379-0148}}
\address{Fujitsu Research, Fujitsu Limited, Kawasaki 211-8588, Japan}
\email{t.shimizu\_2@fujitsu.com}
\author[Norihide Tokushige]{Norihide Tokushige\,\orcid{0000-0002-9487-7545}}
\address{College of Education, University of the Ryukyus, Nishihara  903-0213, Japan}
\email{hide@cs.u-ryukyu.ac.jp}
\urladdr{https://n-tokushige.github.io/x/}

\begin{document}
\begin{abstract}
For an integer $r\geq 3$ and real numbers $a_1,\ldots,a_r\in(0,1)$,
let $a=a_1\cdots a_r$. We show that
\[
\prod_{i=1}^r(a_i+a_i^2+\cdots+a_i^r)\geq
\prod_{i=1}^r(a_i+a_i^2+\cdots+a_i^{r-1}+a).
\]
This inequality enables us to bound the measure of $r$-cross $t$-intersecting families.
\end{abstract}
\maketitle

\section{Introduction}
In this paper, we settle Conjecture~1 of \cite{KKT} and prove the
following inequality. 
\begin{theorem}\label{thm:1}
For an integer $r\geq 3$ and real numbers $a_1,\ldots,a_r\in(0,1)$,
let $a=a_1\cdots a_r$. Then,
\[
\prod_{i=1}^r(a_i+a_i^2+\cdots+a_i^r)\geq
\prod_{i=1}^r(a_i+a_i^2+\cdots+a_i^{r-1}+a).
\]
Moreover, equality holds if and only if $a_1=\cdots=a_r$.
\end{theorem}
\noindent
We mention that this inequality was proved in \cite{KKT} for the case $3\leq r\leq 11$. 
To the best of the authors' knowledge, the stronger Conjecture~2 of \cite{KKT} is still
open.

The inequality in Theorem~\ref{thm:1} is related to a problem in extremal set theory.
For fixed integers $r\geq 2$ and $t\geq 1$, we say that families
$\FF_1,\ldots,\FF_r\subset 2^{[n]}$ are $r$-cross $t$-intersecting
if $|F_1\cap\cdots\cap F_r|\geq t$ for all $F_i\in\FF_i$ ($1\leq i\leq r$),
where $[n]:=\{1,2,\ldots,n\}$. For a fixed real number $p\in(0,1)$, we define
the $\mu_p$-measure of a family $\FF\subset 2^{[n]}$ by
\[
 \mu_p(\FF):=\sum_{F\in\FF}p^{|F|}(1-p)^{n-|F|}.
\]
Using Theorem~\ref{thm:1} we show the following.
\begin{theorem}\label{thm:2}
Let $r\geq 2$ and $n\geq t\geq 1$ be integers, and $p_1,\ldots,p_r\in(0,\frac{r-1}r)$ be 
real numbers. Suppose that $\FF_1,\ldots,\FF_r\subset 2^{[n]}$ are 
$r$-cross $t$-intersecting families. Then,
\[
 \prod_{i=1}^r\mu_{p_i}(\FF_i)\leq (\alpha_1\cdots\alpha_r)^t,
\]
where $\alpha_i$ is the unique root of the equation $(1-p_i)x^r-x+p_i=0$ in $x\in(0,1)$.
\end{theorem}
This result was stated in \cite{T2011} as Lemma~1, but the proof was incomplete. Indeed, the equality in the last line of the proof should be replaced
by an inequality, which follows from Theorem~\ref{thm:1}.
We also make explicit the restriction $0<p_i<1-1/r$,
which is needed for the existence of the root $\alpha_i\in(0,1)$.
For completeness, we include the proof of Theorem~\ref{thm:2} in Section~3.

We derive Theorem~\ref{thm:1} from a stronger inequality. To state the result,
for real numbers $\rho>0$ and $0<x<1$, let
\[
 P_\rho(x):=\frac{1-x^\rho}{1-x},
\]
and set $P_\rho(1):=\rho$. Note that if $\rho=r$ is an integer, then
$P_r(x)=1+x+\cdots+x^{r-1}$.
\begin{theorem}\label{thm:3}
For an integer $r\geq 3$, real numbers $a_1,\ldots,a_r\in(0,1)$ and $\rho\geq r$,
let $G=(a_1\cdots a_r)^{1/r}\in(0,1)$. Then,
\begin{align}\label{eq1}
\prod_{i=1}^r a_i P_\rho(a_i)\geq
\prod_{i=1}^r (P_\rho(a_i)-1+G^\rho).
\end{align}
Moreover, equality holds if and only if $a_1=\cdots=a_r$.
\end{theorem}
Theorem~\ref{thm:1} follows from Theorem~\ref{thm:3} by setting $\rho=r$.

\section{Proof of Theorem~\ref{thm:3}}
We fix an integer $r\geq 3$, real numbers $a_1,\ldots,a_r\in(0,1)$ and 
$\rho\geq r$. Let $G=(a_1\cdots a_r)^{1/r}$.
For $x\in(0,1)$, we have $P_\rho(x)>1$ and $P_\rho(x)-1+G^\rho>G^\rho>0$, and
the inequality \eqref{eq1} is rewritten as
\begin{align}\label{eq2}
 \prod_{i=1}^r
 \frac{P_\rho(a_i)-1+G^\rho}{a_iP_\rho(a_i)}
\leq 1.
\end{align}
To estimate the LHS of \eqref{eq2}, we introduce a function $F(u)$ for
$u<L:=-\log G$ by
\[
 F(u):=\log\frac{P_\rho(Ge^u)+G^\rho-1}{G P_\rho(Ge^u)},
\]
and for each $i$, let
\[
 u_i=\log\frac{a_i}{G},
\]
so that $P_{\rho}(Ge^{u_i})=P_\rho(a_i)$.
Then $u_i=\log a_i-\log G<L$ and 
\begin{align}\label{eq3}
\sum_{i=1}^r u_i=\log\frac{a_1\cdots a_r}{G^r}=0, 
\end{align}
and moreover, 
\[
 \frac{P_\rho(a_i)-1+G^\rho}{a_iP_\rho(a_i)}
 =e^{-u_i+F(u_i)}.
\]
Thus \eqref{eq2} is rewritten as
$\exp\left(\sum_{i=1}^r F(u_i)\right)\leq 1$, 
and our goal is to show that
\begin{align}\label{eq4}
 \sum_{i=1}^r F(u_i)\leq 0, 
\end{align}
and equality holds if and only if $u_1=\cdots=u_r=0$.

We will prove the following two lemmas. 

\begin{lemma}\label{lemma1}
We have $F(0)=0$, and $F$ is concave on $[0,L)$, that is,
$F''(u)\leq 0$ for all $0\leq u<L$.
\end{lemma}

\begin{lemma}\label{lemma2}
We have 
\begin{align}\label{eq5}
mF(u)+F(-mu)\leq 0 
\end{align}
for every $0\leq u<L$ and every integer $m$ with $1\leq m\leq\rho-1$,
where $\rho\geq r\geq 3$.
Moreover, inequality \eqref{eq5} is strict unless $u=0$.
\end{lemma}
These lemmas are consequences of standard calculus.
We defer their proofs until later, and we first finish the proof of \eqref{eq4}
assuming the lemmas.

If $u_i=0$ for all $1\leq i\leq r$, then \eqref{eq4} clearly holds.
Suppose that there are some non-zero terms ($u_i$ such that $u_i\neq 0$).
After deleting zero terms and reindexing, write the remaining tuple as 
$(v_1,\ldots,v_k,-w_1,\ldots,-w_l)$, where $v_i,w_j>0$.
By \eqref{eq3}, let $S:=v_1+\cdots+v_k=w_1+\cdots+w_l>0$.
Since both positive and negative terms occur, we have
$1\leq k\leq r-1\leq\rho-1$.  Moreover, $0<v_iw_j/S\leq v_i<L$ and
$0<w_j/k\leq S/k<L$. Thus every argument used below lies in the required
interval.

It follows from Lemma~\ref{lemma1} that if $x,y\geq 0$ and $0<x+y<L$, then
\[
 F(x)=F\left(\frac x{x+y}\cdot(x+y)+\frac y{x+y}\cdot 0\right)
\geq\frac x{x+y}F(x+y)+\frac y{x+y}F(0)=\frac x{x+y}F(x+y),
\]
and similarly $F(y)\geq\frac y{x+y}F(x+y)$. Therefore, we have
\[
F(x+y) \leq F(x)+F(y).
\]
Using this repeatedly, we get
\[
 F(v_i)=F\left(\frac{v_iw_1}S+\cdots+\frac{v_iw_l}S\right)
\leq F\left(\frac{v_iw_1}S\right)+\cdots+F\left(\frac{v_iw_l}S\right)
\]
for $1\leq i\leq k$. Thus
\[
\sum_{i=1}^k F(v_i)\leq\sum_{i=1}^k \sum_{j=1}^l F\left(\frac{v_iw_j}S\right)
=\sum_{j=1}^l \sum_{i=1}^k F\left(\frac{v_iw_j}S\right).
\] 
By concavity, it follows that
\begin{align*}
\sum_{i=1}^k F\left(\frac{v_iw_j}S\right)
&=k\left(\frac1k F\left(\frac {v_1w_j}S\right)+\cdots
+\frac1k F\left(\frac {v_kw_j}S\right)\right) \\
&\leq k F\left(\frac{v_1w_j}{kS} +\cdots+\frac {v_kw_j}{kS}\right) 
=k F\left(\frac{w_j}k\right).
\end{align*}
Consequently, we have
\[
\sum_{i=1}^r F(u_i)=\sum_{i=1}^k F(v_i)+\sum_{j=1}^lF(-w_j)\leq
\sum_{j=1}^l\left( k F\left(\frac{w_j}k\right) + F(-w_j)\right)\leq0,
\]
where the last inequality follows from Lemma~\ref{lemma2} with $m=k$ and 
$u=\frac{w_j}k$.

If at least one $u_i$ is nonzero, then $k,l>0$, and $w_j/k>0$ 
for all $j$.
Thus each application of Lemma~\ref{lemma2} is strict, and 
$\sum_{i=1}^rF(u_i)<0$. 
Therefore $\sum_{i=1}^rF(u_i)=0$ is possible only when all $u_i=0$,
that is, $a_i=G$.
Consequently, equality holds in \eqref{eq2} 
(and so in \eqref{eq1}) only if $a_1=\cdots=a_r$. The converse is trivial,
and this completes the proof of Theorem~\ref{thm:3} assuming the two lemmas.

\medskip
Now we prove the two lemmas.

\begin{proof}[Proof of Lemma~\ref{lemma1}]
Since $P_\rho(G)+G^\rho-1=GP_\rho(G)$, it follows that $F(0)=0$.

For $0\leq u<L$, let $x=Ge^u\in[G,1)$, $c=1-G^\rho\in(0,1)$, and 
$H(u)=\log P_\rho(x)$. Since $P_\rho(x)\geq 1$, $H(u)$ is well-defined.
We claim that 
\begin{align}\label{eq6}
(H')^2>xH''. 
\end{align}
Noting that $\frac{dx}{du}=x$ and $H'=\frac d{du}H(u)=\frac{dx}{du}\frac d{dx}\log P_\rho(x)$, we have
\[
 H'=\frac{x}{1-x}-\frac{\rho x^\rho}{1-x^\rho},\quad
 H''=\frac{x}{(1-x)^2}
       -\frac{\rho^2x^\rho}{(1-x^\rho)^2},
\]
and
\begin{align}\label{eq7}
 (H')^2-xH''
 =\frac{\rho x^{\rho+1}}{(1-x^\rho)^2}
\left(\rho(1+x^{\rho-1})-2P_\rho(x)\right).
\end{align}
We use a simple fact that $f(t):=t^{\rho-1}$ is strictly convex (because 
$\rho\geq r\geq 3$), and so, for $x<1$, we obtain
\[
\frac{1-x^\rho}{\rho}=\int_x^1f(t)dt<\frac{1-x}2(f(1)+f(x))
=\frac{1-x}2(1+x^{\rho-1}),
\]
that is, $2P_\rho(x)<\rho(1+x^{\rho-1})$. Substituting this into
\eqref{eq7}, we get $(H')^2-xH''>0$.

Recall that 
 \[
F(u)=\log(P_\rho(x)-c)-\log P_\rho(x)-\log G.
 \]
Some direct computation shows
\[
 F'=\frac{cH'}{P_\rho(x)-c},\quad
 F''=\frac{c}{(P_\rho(x)-c)^2}
 \left((P_\rho(x)-c)H''-P_\rho(x)(H')^2\right).
\]
Using $P_\rho(x)-c>0$ and \eqref{eq6}, we have
\begin{align*}
(P_\rho(x)-c)H''-P_\rho(x)(H')^2< \left(\frac{P_\rho(x)-c}{x}-P_\rho(x)
\right)(H')^2=\frac{G^\rho-x^\rho}{x}(H')^2\leq0.
\end{align*}
This yields  $F''(u)<0$. 
\end{proof}

\begin{proof}[Proof of Lemma~\ref{lemma2}]
If $u=0$, then \eqref{eq5} clearly holds with equality.
So we may assume that $u>0$.
For $t>0$, let
\[
 R_\rho(t):=\frac{t^{1-\rho}-1}{1-t},
\]
and set $R_\rho(1):=\rho-1$ so that $R_\rho(t)$ is continuous at $t=1$. 
Let $\lambda=e^{-\rho u}$.
We claim that \eqref{eq5} follows from
\begin{align}\label{eq8}
 \frac{R_\rho(Ge^{-mu})}{R_\rho(Ge^u)}\geq
 \frac{R_{m+1}(\lambda)}{R_{m+1}(1)}.
\end{align}
To this end, let 
$p=R_\rho(Ge^u)$, $q=R_\rho(Ge^{-mu})$.
A direct computation shows that
 \[
 e^{F(u)}=e^u\frac{p+\lambda}{p+1},
 \quad
 e^{F(-mu)}=e^{-mu}\frac{q+\lambda^{-m}}{q+1}.
\]
Set $\theta:=\frac{p+\lambda}{p+1}$. Then $0<\lambda<\theta<1$.
Since $e^{mF(u)+F(-mu)} =\theta^m  \frac{q+\lambda^{-m}}{q+1}$, we can rewrite
\eqref{eq5} as $\theta^m  \frac{q+\lambda^{-m}}{q+1}\leq 1$,
or equivalently,
\[
 \frac qp\geq\frac{\lambda^{-m}\theta^m-1}{p(1-\theta^m)}
=\left(\sum_{i=1}^m\theta^{i-1}\lambda^{-i}\right)\bigg/
\left(\sum_{i=1}^m\theta^{i-1}\right).
\]
The RHS is at most $\frac1m\sum_{j=1}^m\lambda^{-j}$, because
\begin{align*}
\left(\sum_{i=1}^m\theta^{i-1}\right)
\left(\sum_{j=1}^m\lambda^{-j}\right)
-m\sum_{i=1}^m\theta^{i-1}\lambda^{-i}
=\sum_{1\leq i<j\leq m}
 (\theta^{i-1}-\theta^{j-1})(\lambda^{-j}-\lambda^{-i})\geq 0,
\end{align*}
see, e.g., Theorem~43 (Tchebychef's inequality) in \cite{HLP}.
Hence, to show \eqref{eq5}, it suffices to show that
\[
 \frac qp\geq
\frac1m\sum_{j=1}^m\lambda^{-j}=\frac{\lambda^{-m}-1}{m(1-\lambda)}
= \frac{R_{m+1}(\lambda)}{R_{m+1}(1)},
\]
which is \eqref{eq8}.

\medskip
To prove \eqref{eq8}, we introduce two auxiliary functions $\Phi$ and $Q_\rho$. For the proof, it is important that $\Phi$ is an increasing function, 
and $Q_\rho$ is a convex function. For $t>0$, let
\[
 \Phi(t):=\frac{\sinh t}{t},
\]
and set $\Phi(0):=1$ so that $\Phi(t)$ is continuous at $t=0$. 
Then, $\Phi$ is strictly increasing for $t>0$, because
\[
 \Phi'(t)=\frac{t\cosh t-\sinh t}{t^2}
 =\frac{1}{t^2}\int_0^t z\sinh z\,dz>0.
\]
For $t>0$, we have
\begin{align}\label{eq9}
 R_\rho(e^{-t})
& =\frac{e^{(\rho-1)t}-1}{1-e^{-t}}
 =\frac{e^{(\rho-1)t/2}}{e^{-t/2}}\cdot
\frac{e^{(\rho-1)t/2}-e^{-(\rho-1)t/2}}{e^{t/2}-e^{-t/2}}
 =e^{\rho t/2}\frac{\sinh((\rho-1)t/2)}{\sinh(t/2)}.
\end{align}
More generally, for $\alpha>1$ and $\beta>0$, we have
\begin{align}\label{eq10}
R_\alpha(e^{-\beta})=\frac{e^{\frac{(\alpha-1)\beta}2}
\left(e^{\frac{(\alpha-1)\beta}2}-e^{-\frac{(\alpha-1)\beta}2}\right)}
{e^{-\frac{\beta}2}\left(e^{\frac{\beta}2}-e^{-\frac{\beta}2}\right)}
=(\alpha-1)e^{\frac{\alpha\beta}2}\,
\frac{\Phi((\alpha-1)\beta /2)}{\Phi(\beta /2)}.
\end{align}

For $t>0$, let $Q_{\rho}(t):=\log R_\rho(e^{-t})$, and set
$Q_\rho(0):=\log(\rho-1)$ so that $Q_\rho(t)$ is continuous at $t=0$.
For $t>0$, it follows from \eqref{eq9} that 
\begin{align*}
Q_\rho(t) &=\frac{\rho}2t+\log(\sinh((\rho-1)t/2))-\log(\sinh(t/2)),\\ 
Q'_\rho(t) &=\frac12\left(\rho-\coth(t/2)+(\rho-1)\coth((\rho-1)t/2)\right),\\
Q_{\rho}''(t)
&=\frac14\left(\frac1{\sinh^2(t/2)}-\frac{(\rho-1)^2}{\sinh^2((\rho-1)t/2)}\right)
 =\frac1{t^2}\left(\frac1{\Phi(t/2)^2}-\frac1{\Phi((\rho-1)t/2)^2}\right)>0.
\end{align*}
All expressions at $t=0$ are interpreted by continuity.
We also have $\lim_{t\to 0}Q''_\rho(t)=\rho(\rho-2)/12>0$.
Thus, $Q''_\rho(t)>0$ for all $t\geq 0$.
By convexity, for any $t>0$ and $a>0$, we have
\[
 Q_{\rho}(t+a)-Q_{\rho}(t)
 >Q_{\rho}(a)-Q_{\rho}(0).
\]
Substituting $t=L-u$ and $a=(m+1)u$, we get
\[
 Q_\rho(L+mu)-Q_\rho(L-u)=Q_\rho((L-u)+(m+1)u)-Q_\rho(L-u)
>Q_\rho((m+1)u)-Q_\rho(0).
\]
Exponentiating both sides, and using $G=e^{-L}$, it follows that
\begin{align}\label{eq11}
 \frac{R_\rho(Ge^{-mu})}{R_\rho(Ge^{u})}=
 \frac{R_\rho(e^{-(L+mu)})}{R_\rho(e^{-(L-u)})}
> \frac{R_\rho(e^{-(m+1)u})}{R_\rho(1)}.
\end{align}

For $\sigma\geq m+1$ and fixed $d:=(m+1)\rho u/2$, let
\[
 H_d(\sigma):=\frac{R_\sigma(e^{-2d/\sigma})}{R_\sigma(1)}
=e^d\frac{\Phi((\sigma-1)d/\sigma)}{\Phi(d/\sigma)},
\]
where we used \eqref{eq10} in the last equality.
Recall that $\Phi(t)$ is increasing in $t$. Thus, $\Phi((\sigma-1)d/\sigma)$
is increasing in $\sigma$ and $\Phi(d/\sigma)$ is decreasing in $\sigma$, 
and so $H_d(\sigma)$ is increasing in $\sigma$.
Therefore, using $\rho\geq m+1$, we have
\begin{align}\label{eq12}
\frac{R_\rho(e^{-(m+1)u})}{R_\rho(1)}=H_d(\rho)\geq H_d(m+1)=
\frac{R_{m+1}(e^{-\rho u})}{R_{m+1}(1)}
 =\frac{R_{m+1}(\lambda)}{R_{m+1}(1)}.
\end{align}
Finally, by \eqref{eq11} and \eqref{eq12}, we obtain
\[
 \frac{R_\rho(Ge^{-mu})}{R_\rho(Ge^u)}
 >\frac{R_\rho(e^{-(m+1)u})}{R_\rho(1)}
 \geq\frac{R_{m+1}(\lambda)}{R_{m+1}(1)}.
\]
This proves \eqref{eq8}, which yields \eqref{eq5} without equality.

If $u=0$, then equality holds in \eqref{eq5}.
This completes the proof of Lemma~\ref{lemma2}.
\end{proof}

\section{Proof of Theorem~\ref{thm:2}}
For $1\leq i\leq r$, let $q_i:=1-p_i\in(\frac1r,1)$. Let
\[
 f(x):=-x+\prod_{i=1}^r(p_i+q_ix).
\]
\begin{lemma}\label{lemma3}
 We have $f(0)>0$, $f(1)=0$, $f'(1)>0$, and $f''(x)>0$ for $x>0$.
\end{lemma}
\begin{proof}
We have $f(0)=\prod_{i=1}^rp_i>0$, $f(1)=-1+\prod_{i=1}^r(p_i+q_i)=0$,
and $f'(x)=-1+\sum_{i=1}^rq_i\prod_{j\neq i}(p_j+q_jx)$.
Since $q_i>\frac1r$, it follows that $f'(1)=-1+\sum_{i=1}^rq_i>0$.
Also, we have $f''(x)=\sum_{i=1}^r q_i\sum_{j\neq i}q_j\prod_{k\neq i,j}(p_k+q_kx)>0$ for $x>0$.
\end{proof}
Using Lemma~\ref{lemma3}, the graph of $y=f(x)$ yields the following.
\begin{lemma}\label{lemma4}
There is a unique $\beta\in(0,1)$ such that $f(\beta)=0$.
If $0<\alpha<1$ and $f(\alpha)\leq0$, then $\beta\leq\alpha$.
\end{lemma}

We mention that $g(x):=-x+p_i+q_ix^r$ satisfies $g(0)>0$, $g(1)=0$,
$g'(1)>0$, and $g''(x)>0$ for $x>0$. Thus, there is a unique
$\alpha_i\in(0,1)$ such that $g(\alpha_i)=0$.
Since $-\alpha_i+p_i+(1-p_i)\alpha_i^r=0$, it follows that
\[
p_i=\frac{\alpha_i-\alpha_i^r}{1-\alpha_i^r}, \quad
q_i=\frac{1-\alpha_i}{1-\alpha_i^r}.
\]
Let $\alpha=\alpha_1\cdots\alpha_r\in(0,1)$. Then we have
\begin{align*}
&q_i(\alpha_i+\cdots+\alpha_i^r)=\alpha_i,\\
&q_i(\alpha_i+\cdots+\alpha_i^{r-1}+\alpha)=p_i+q_i\alpha.
\end{align*}

The following result was stated in \cite{T2024} as Theorem~10,
but the proof of Claim 1 in that paper contains a correctable error.
We will give a slightly different self-contained proof in the next section.
\begin{proposition}\label{prop}
We have $\prod_{i=1}^r\mu_{p_i}(\FF_i)\leq\beta^t$.
\end{proposition}

Our goal is to show that $\prod_{i=1}^r\mu_{p_i}(\FF_i)\leq\alpha^t$.
To this end, by Proposition~\ref{prop}, it suffices to show $\beta\leq\alpha$,
and by Lemma~\ref{lemma4}, it suffices to show the following.

\begin{lemma}\label{lemma6}
We have $f(\alpha)\leq 0$.
\end{lemma}
\begin{proof}
This is a consequence of Theorem~\ref{thm:1}. 
If $r\geq 3$, then multiplying both sides of the inequality of 
Theorem~\ref{thm:1} by
$\prod_{i=1}^rq_i>0$, we obtain 
\begin{align*}
\alpha&=\prod_{i=1}^rq_i(\alpha_i+\cdots+\alpha_i^r)
\geq\prod_{i=1}^r q_i(\alpha_i+\cdots+\alpha_i^{r-1}+\alpha)
=\prod_{i=1}^r(p_i+q_i\alpha),
\end{align*}
which means that $f(\alpha)\leq 0$.
If $r=2$, then 
$(\alpha_1+\alpha_1^2)(\alpha_2+\alpha_2^2)=(\alpha_1+\alpha_1\alpha_2)
(\alpha_2+\alpha_1\alpha_2)$. Thus $f(\alpha)=0$.
\end{proof}

Combining Proposition~\ref{prop}, Lemma~\ref{lemma4}, and Lemma~\ref{lemma6}, we obtain
\[
 \prod_{i=1}^r\mu_{p_i}(\FF_i)\leq\beta^t\leq\alpha^t=(\alpha_1\cdots\alpha_r)^t.
\]
This completes the proof of Theorem~\ref{thm:2}, subject to Proposition~\ref{prop}, which is proved in the next section.

\section{Proof of Proposition~\ref{prop}}
We recall some basic facts concerning shifting operations and shifted intersecting families.
For $\FF\subset 2^{[n]}$ and $1\leq a<b\leq n$, define a shifting operation $S_{a,b}$ by
\[
 S_{a,b}(\FF):=\{s_{a,b,\FF}(F):F\in\FF\}\subset 2^{[n]},
\]
where $s_{a,b,\FF}(F)$ is defined as follows: if $F\cap\{a,b\}=\{b\}$ and
$F':=(F\setminus\{b\})\cup\{a\}\not\in\FF$ then $s_{a,b,\FF}(F):=F'$, otherwise 
$s_{a,b,\FF}(F):=F$. 
The map $s_{a,b,\FF}$ is a size-preserving bijection
from $\FF$ onto $S_{a,b}(\FF)$. Hence $\mu_p(\FF)=\mu_p(S_{a,b}(\FF))$.

Suppose that $\FF_1,\ldots,\FF_r\subset 2^{[n]}$ are $r$-cross $t$-intersecting.
Let $\GG_i:=S_{a,b}(\FF_i)$ for $i\in[r]$. Then, $\GG_1,\ldots,\GG_r$ are
$r$-cross $t$-intersecting as well. To see this, suppose the contrary, and choose
$F_i\in\FF_i$ for $i\in[r]$ such that $|F_1\cap\cdots\cap F_r|=t$ and
$F_1\cap\cdots\cap F_r\cap\{a,b\}=\{b\}$, while 
$G_1\cap\cdots\cap G_r\cap\{a,b\}=\emptyset$, where $G_i=s_{a,b,\FF_i}(F_i)$. 
Without loss of generality, we may assume that $G_1\cap\{a,b\}=\{b\}$. This happens
because there is $F'\in\FF_1$ such that $F'=(F_1\setminus\{b\})\cup\{a\}$.
But then $F'\cap F_2\cap\cdots\cap F_r\cap\{a,b\}=\emptyset$
and $|F'\cap F_2\cap\cdots\cap F_r|<t$, a contradiction.

We say that a family $\FF\subset 2^{[n]}$ is shifted if $S_{a,b}(\FF)=\FF$ for all
$1\leq a<b\leq n$.
For given families $\FF_1,\ldots,\FF_r\subset 2^{[n]}$, 
we can make all of them shifted by applying finitely many simultaneous shifting 
operations. Indeed, whenever a simultaneous shift changes at least one family, 
the nonnegative integer $\sum_{i=1}^r\sum_{F\in\FF_i}\sum_{j\in F}j$ strictly decreases.

We say that families $\AA_1,\ldots,\AA_r\subset 2^{[n]}$ satisfy the $(r,t)$-excess 
condition if for every $A_i\in\AA_i$ ($i\in[r]$), there exists $k\in[n]$ such that
\[
 \sum_{i=1}^r|A_i\cap[k]|-(r-1)k\geq t.
\]
\begin{lemma}\label{lemma7}
Let $\AA_1,\ldots,\AA_r\subset 2^{[n]}$ be shifted families. Then, they satisfy
the $(r,t)$-excess condition if and only if they are $r$-cross $t$-intersecting.
\end{lemma}
For the proof, we use the following simple fact. For any subsets 
$B_1,\ldots,B_r\subset[n]$ and $k\in[n]$, it follows that
\begin{align}\label{eq13}
 \sum_{i=1}^r|B_i\cap[k]|\leq (r-1)k+\left|[k]\cap\bigcap_{i=1}^rB_i\right|. 
\end{align}
Indeed, in the LHS, if $x\in[k]\cap\bigcap_{i=1}^rB_i$, then it is counted $r$ times,
otherwise $x$ is counted at most $r-1$ times.

\begin{proof}[Proof of Lemma~\ref{lemma7}]
First suppose that the families $\AA_i$ satisfy the $(r,t)$-excess condition. Then,
for every $A_i\in\AA_i$, there is $k\in[n]$ such that
\begin{align*}
 t&\leq\sum_{i=1}^r|A_i\cap[k]|-(r-1)k
\leq\left|[k]\cap \bigcap_{i=1}^rA_i\right|
\leq\left|\bigcap_{i=1}^rA_i\right|,
\end{align*}
where the second inequality follows from \eqref{eq13}.
This means that the families are $r$-cross $t$-intersecting.

Next suppose that the families are $r$-cross $t$-intersecting, but they do not
satisfy the $(r,t)$-excess condition. We call $(A_1,\ldots,A_r)$, where $A_i\in\AA_i$,
a witness if 
\[
 \sum_{i=1}^r|A_i\cap[k]|-(r-1)k<t 
\]
for all $k\in[n]$. Among witnesses, we choose one so that $|C|$ is minimal, 
where $C:=\bigcap_{i=1}^rA_i$.
Since the families are $r$-cross $t$-intersecting, it follows that $|C|\geq t$, and
let $b$ be the $t$-th element of $C$, that is, $|C\cap[b]|=t$.
Then, 
$\sum_{i=1}^r|[b]\setminus A_i|=b-\left(\sum_{i=1}^r|A_i\cap[b]|-(r-1)b\right)
>b-t=\left|\bigcup_{i=1}^r([b]\setminus A_i)\right|$, and so
there is $a<b$ such that at most $r-2$ of $A_1,\ldots,A_r$ contain $a$.
Without loss of generality, we may assume that $a\not\in A_1\cup A_2$. 
By the shiftedness of the families, we have $B_1:=(A_1\setminus\{b\})\cup\{a\}\in\AA_1$,
and let $B_i:=A_i$ for $2\leq i\leq r$. Let $C'=\bigcap_{i=1}^r B_i=C\setminus\{b\}$.
We show that $(B_1,\ldots,B_r)$ is also a witness.
If $k<b$, then, by \eqref{eq13}, 
$\sum_{i=1}^r|B_i\cap[k]|-(r-1)k\leq|[k]\cap C'|=|[k]\cap C|<|[b]\cap C|=t$.
If $k\geq b$, then, 
$\sum_{i=1}^r|B_i\cap[k]|-(r-1)k=\sum_{i=1}^r|A_i\cap[k]|-(r-1)k<t$.
Therefore, $(B_1,\ldots,B_r)$ is a witness, but $|C'|<|C|$ contradicts the minimality
of $|C|$.
\end{proof}

\begin{proof}[Proof of Proposition~\ref{prop}]
Let $\FF_1,\ldots,\FF_r\subset 2^{[n]}$ be $r$-cross $t$-intersecting families. 
If one of the families is empty, the assertion is immediate.
Otherwise, simultaneous shifting and Lemma~\ref{lemma7} give shifted
families $\AA_1,\ldots,\AA_r\subset 2^{[n]}$ satisfying the $(r,t)$-excess condition 
and $\mu_{p_i}(\AA_i)=\mu_{p_i}(\FF_i)$ for every $i\in[r]$.
Let $q_i:=1-p_i$ and 
\[
 \phi(z):=\prod_{i=1}^r(p_i+q_iz)=\sum_{j=0}^rc_jz^j.
\]
Let $\beta\in(0,1)$ be defined by $\phi(\beta)=\beta$.

Let $X_{i,k}\in\{0,1\}$ ($i\in[r]$, $k\in[n]$) be mutually independent random variables
with $\P[X_{i,k}=1]=p_i$.
For every $i\in[r]$, sample a random set $Z_i=\{k\in[n]:X_{i,k}=1\}$. Then 
$\P[Z_i\in\AA_i]=\mu_{p_i}(\AA_i)$, and
\[
 \P[Z_i\in\AA_i\text{ for every }i\in[r]]=\prod_{i=1}^r\mu_{p_i}(\AA_i).
\]
Define independent and identically distributed random variables $Y_k$ for $k\in[n]$ by
\[
 Y_k:=\sum_{i=1}^r(1-X_{i,k}), 
\]
and set $S_0:=0$, $S_k:=\sum_{j=1}^k(1-Y_j)$ for $k\in[n]$. 
Note that 
\[
Y_k=\#\{i\in[r]:k\not\in Z_i\}=r-\#\{i\in[r]:k\in Z_i\}, 
\]
and 
\[
 S_k=\sum_{j=1}^k(1-r+\#\{i\in[r]:j\in Z_i\})=\sum_{i=1}^r|Z_i\cap[k]|-(r-1)k.
\] 
Since $\AA_1,\ldots,\AA_r$ satisfy the $(r,t)$-excess condition, it follows that
\[
 \prod_{i=1}^r\mu_{p_i}(\AA_i)=\P[Z_i\in\AA_i\text{ for every }i]\leq\P[\max_{0\leq k\leq n} S_k\geq t].
\]
Note also that $X_{i,k}$ are independent and
\[
 \E[z^{Y_k}]=\E[z^{\sum_{i=1}^r(1-X_{i,k})}]=\prod_{i=1}^r\E[z^{1-X_{i,k}}]
=\prod_{i=1}^r(p_i+q_iz)=\phi(z).
\]
Thus, $\E[z^{Y_k}]=\phi(z)=\sum_{j=0}^rc_jz^j$, and so 
\[
c_j=\P[Y_k=j]. 
\]

For integers $0\leq M\leq n$ and $d\geq 0$, define
\[
 h_M(d):=\P[\max_{0\leq k\leq M}S_k\geq d].
\]
We have $h_M(0)=1$ and $h_0(d)=0$ for $d\geq 1$.
We claim that, for $d\geq 1$, 
\begin{align}\label{eq14}
 h_{M+1}(d)=\sum_{j=0}^rc_jh_M(d-1+j).
\end{align}
For $0\leq M<n$, let $\tilde S_k:=S_{k+1}-S_1=\sum_{l=2}^{k+1}(1-Y_l)$.
Since $Y_1,\ldots,Y_{M+1}$ are i.i.d.\ variables, 
$(\tilde S_0,\ldots,\tilde S_M)$ has the same
distribution as $(S_0,\ldots,S_M)$, and is independent of $Y_1$.

Conditioning on $Y_1=j$, we have
\begin{align}\label{eq15}
 h_{M+1}(d)=\sum_{j=0}^r\P[Y_1=j]\,
\P[\max_{0\leq k\leq M+1}S_k\geq d\mid Y_1=j].
\end{align}
For $Y_1=j$, we have
\[
 \max_{1\leq k\leq M+1}S_k=(1-j)+\max_{0\leq k\leq M}\tilde S_k.
\]
Noting that $S_0=0<d$, and $d-1+j\geq 0$, it follows that 
\begin{align}\label{eq16}
 \P[\max_{0\leq k\leq M+1}S_k\geq d\mid Y_1=j]
= \P[\max_{0\leq k\leq M}\tilde S_k\geq d-1+j\mid Y_1=j]
=h_M(d-1+j).
\end{align}
Then \eqref{eq14} follows from \eqref{eq15}, \eqref{eq16}, and $\P[Y_1=j]=c_j$.

Finally we show that
\begin{align}\label{eq17}
h_M(d)\leq \beta^d  
\end{align}
by induction on $M$, simultaneously for all integers $d\geq0$.
The assertion is immediate for $M=0$ and for $d=0$.
For $d\geq1$, \eqref{eq14} and the induction hypothesis give
\begin{align*}
h_{M+1}(d) 
&=\sum_{j=0}^r c_jh_M(d-1+j)
\leq\sum_{j=0}^r c_j \beta^{d-1+j}
=\beta^{d-1}\sum_{j=0}^r c_j \beta^{j}
=\beta^{d}.
\end{align*}
Setting $M=n$ and $d=t$ in \eqref{eq17}, we conclude that
\[
 \prod_{i=1}^r\mu_{p_i}(\FF_i)
 =\prod_{i=1}^r\mu_{p_i}(\AA_i)
 \leq h_n(t)\leq\beta^t.
\]
This completes the proof of Proposition~\ref{prop}.
\end{proof}

\subsection*{The role of AI in this paper}
AI assistance (GPT-5.5, GPT-5.6, and GPT-6) was used as a research support tool for 
exploring possible proofs.
After this exploratory stage, all mathematical arguments were independently 
verified and refined by the authors. 
The final manuscript was written by the authors.

\subsection*{Acknowledgments}
NT was supported by JSPS KAKENHI Grant Number JP23K03201.

\end{document}